\documentclass[preprint,12pt]{elsarticle}
\usepackage{amssymb}
\usepackage{amsmath}
\usepackage{amsfonts,amsthm,latexsym,mathrsfs}
\usepackage{color}

\theoremstyle{plain}
\newtheorem{theorem}[subsection]{{\bf Theorem}}
\newtheorem{corollary}[subsection]{{\bf Corollary}}
\newtheorem{proposition}[subsection]{{\bf Proposition}}
\newtheorem{lemma}[subsection]{{\bf Lemma}}

\theoremstyle{remark}
\newtheorem{remark}[subsection]{{\it Remark}}
\newtheorem{example}[subsection]{{\it Example}}
\numberwithin{equation}{subsection}

\def \d {\mathrm d}

\def\X{\mathfrak C}
\def\Y{\mathfrak A}

\newcommand{\Aut}{\mathrm{Aut}}

\def \Z {\mathbb Z}

\def \Aut{ \mathrm {Aut}}

\newcounter{ithmcount}

\newenvironment{ithm}{\begin{list}{{\rm \alph{ithmcount})}}{\usecounter{ithmcount}\labelwidth18pt
      \leftmargin18pt \topsep3pt \itemsep1pt \parsep2pt}}{\end{list}}

\makeatletter
\def\@author#1{\g@addto@macro\elsauthors{\normalsize%
    \def\baselinestretch{1}%
    \upshape\authorsep#1\unskip\textsuperscript{%
      \ifx\@fnmark\@empty\else\unskip\sep\@fnmark\let\sep=,\fi
      \ifx\@corref\@empty\else\unskip\sep\@corref\let\sep=,\fi
      }%
    \def\authorsep{\unskip,\space}%
    \global\let\@fnmark\@empty
    \global\let\@corref\@empty 
    \global\let\sep\@empty}%
    \@eadauthor={#1}
}
\makeatother

\journal{Journal of Algebra}

\begin{document}

\begin{frontmatter}

\title{Groups in which every non-abelian subgroup is self-centralizing}
\author[CD]{Costantino Delizia\corref{cor1}}
\ead{cdelizia@unisa.it}

\author[HD]{Heiko Dietrich\fnref{ackHD}}
\ead{heiko.dietrich@monash.edu}

\author[PM]{Primo\v z Moravec\fnref{ackPM}}
\ead{primoz.moravec@fmf.uni-lj.si}

\author[CD]{Chiara Nicotera}
\ead{cnicoter@unisa.it} 

\cortext[cor1]{Corresponding author }
\fntext[ackHD]{Dietrich was supported by an ARC DECRA (Australia), project DE140100088, and  by a Go8-DAAD Joint Research 
Co-operation Scheme, project ``Groups of Prime-Power Order and Coclass  Theory''.}
\fntext[ackPM]{Moravec was supported by ARRS (Slovenia), projects P1-0222 and J1-5432.}

\address[CD]{University of Salerno, Italy}
\address[HD]{Monash University, Melbourne, Australia}
\address[PM]{University of Ljubljana, Slovenia}

\date{}

\begin{abstract}
We study groups having the property that every non-abelian subgroup contains its centralizer. We describe various classes of infinite groups in this class, and address a problem of Berkovich regarding the classification of finite $p$-groups with the above property.
\end{abstract}

\begin{keyword}
centralizer\sep non-abelian subgroup\sep self-centralizing subgroup
\end{keyword}

\end{frontmatter}

\section{Introduction}
\label{s:intro}

A subgroup $H$ of a group $G$ is {\it self-centralizing} if the centralizer $C_G(H)$ is contained in $H$. Clearly, an abelian subgroup $A$ of $G$ is self-centralizing if and only if $C_G(A)=A$. In particular, the trivial subgroup of $G$ is self-centralizing if and only if $G$ is trivial. 

The structure of groups in which many non-trivial subgroups are self-centralizing has been studied in several papers.
In \cite{Del13} it has been proved  that a locally graded group (that is, a group in which every non-trivial finitely generated subgroup has a proper subgroup of finite index) in which all non-trivial subgroups are self-centralizing has to be finite; therefore it has to be  cyclic of prime order or a non-abelian group whose order is a product of two different primes. 
The papers \cite{Del13} and \cite{Del15} deal with the class $\X$ of  groups in which every non-cyclic subgroup is self-centralizing; in particular, a complete classification of locally finite $\X$-groups is given.

In this paper, we study the class $\Y$ of  groups in which every non-abelian subgroup is self-centralizing.
We note that the class $\Y$ is fairly wide. Clearly, it contains all $\X$-groups. It also contains the class of commutative-transitive groups (that is, groups in which the centralizer of each non-trivial element is abelian), see \cite{Wu98}. Moreover, by definition, the class $\Y$ contains all minimal non-abelian groups (that is, non-abelian groups in which every proper subgroup is abelian); in parti\-cu\-lar, Tarski monsters are $\Y$-groups. 

The structure and main results of the paper are as follows.
In Section~\ref{s:basic} we derive some basic properties of $\Y$-groups; these results are crucial for the further investigations in the subsequent sections. In Section~\ref{s:nilpotent} we consider infinite nilpotent $\Y$-groups; for example, we prove that such groups are abelian, which reduces the investigation of nilpotent $\Y$-groups to finite $p$-groups in $\Y$. Infinite supersoluble groups in $\Y$ are classified in Section~\ref{s:supersoluble}; for example, we prove that  if such a group has no element of order~2, then it must be abelian. In Section~\ref{s:soluble} we discuss some properties of soluble groups in $\Y$. Lastly, in Section~\ref{s:finite}, we consider finite $\Y$-groups, and we derive various characterisations of finite groups in $\Y$.
Motivated by Section~\ref{s:nilpotent}, we focus on finite $p$-groups in $\Y$;  Problem 9 of \cite{Ber09} asks for a classification of such groups. This appears to be hard, as there seem to be many classes of finite $p$-groups that belong to $\Y$.
We show that all finite metacyclic $p$-groups are in $\Y$, and classify the finite $p$-groups in $\Y$ which have maximal class or exponent $p$.

\section{Basic properties of $\Y$-groups}
\label{s:basic}
We collect some basic properties of $\Y$-groups. Since every free group lies in  $\Y$, the class $\Y$ is not quotient closed. On the other hand, $\Y$ obviously is subgroup closed. Similarly, the next lemma is an easy observation.
 
\begin{lemma}
\label{l:center}
If $G$ is an $\Y$-group, then its center $Z(G)$ is contained in every non-abelian subgroup of $G$.
\end{lemma}

As usual, we denote by $\Phi(G)$ the Frattini subgroup of a group $G$.

\begin{lemma}
\label{l:frattini}
If $G\in\Y$ is non-abelian group, then $Z(G)\le\Phi(G)$.
\end{lemma}
\begin{proof}
Let $M$ be a maximal subgroup of $G$.  If $M$ is abelian and $Z(G)\not\le M$, then $MZ(G)=G$, hence $G$ is abelian, a contradiction. On the other hand, if $M$ is non-abelian, then $Z(G)\le C_G(M)< M$. In conclusion, $Z(G)$ lies in every maximal subgroup of $G$, thus $M\leq\Phi(G)$.
\end{proof}

\begin{lemma}
\label{l:normal}
If $G\in \Y$ is infinite, then every normal subgroup of $G$ is abelian or infinite.
\end{lemma}
\begin{proof} Let $H\unlhd G$ be non-abelian. If $H$ is finite, then so is $C_G(H)< H$, and $|G|=[G:C_G(H)||C_G(H)|\leq |\Aut(H)||C_G(H)|<\infty$,  a contradiction.
\end{proof}

The next three results will be of particular importance in Section~\ref{s:supersoluble}. Recall that a group element is aperiodic if its order is infinite.

\begin{lemma}
\label{l:aperiodic}
Let $G\in \Y$. If $N$ is a finite normal subgroup of $G$, then $C_G(N)$ contains all aperiodic elements of $G$.
\end{lemma}
\begin{proof} Let $x$ be an aperiodic element of $G$. Suppose, for a contradiction, that $x\notin C_G(N)$. Since $G/C_G(N)$ embeds into the finite group $\Aut(N)$, it is finite. Thus,  there exists a positive integer $n$ such that $x^n\in C_G(N)$. Let $p$ be a prime number which does not divide $n$, so that $x^p\notin C_G(N)$. Then  $N\langle x^p\rangle$ is non-abelian, hence $C_G(N\langle x^p\rangle)< N\langle x^p\rangle$. Since $x^n\in C_G(N\langle x^p\rangle)$, there exist $s\in N$ and $\alpha\in \Z$ such that $x^n=s(x^p)^\alpha$. Now $s=x^{n-p\alpha}\in N\cap\langle x\rangle=\{1\}$ yields $x^{n-p\alpha}=1$, hence $n=p\alpha$, contradicting our choice of $p$.
\end{proof}

\begin{corollary}
\label{c:aperiodic}
Let $G\in \Y$. If $G$ is generated by aperiodic elements, then every finite normal subgroup of $G$ is central.
\end{corollary}

\begin{lemma}
\label{l:inverting}
Let $G\in \Y$ and let $a\in G$ be non-trivial of infinite or odd order. If $a^x=a^{-1}$ for some $x\in G$, then $x$ is periodic and its order is a power of $2$. Moreover, $C_G(\langle a,x\rangle)=\langle x^2\rangle$.
\end{lemma}
\begin{proof} 
If $x$ is aperiodic, then $\langle a, x^3\rangle$ is non-abelian, hence $C_G(\langle a, x^3\rangle)< \langle a, x^3\rangle = \langle a\rangle\langle x^3\rangle$. Since $x^2\in C_G(\langle a,x^3\rangle)$, we can write $x^2= a^\alpha (x^3)^\beta$ for some $\alpha,\beta\in \Z$. Since $x^2$ commutes with $x$, we must have  $a^\alpha=a^{-\alpha}$, and  $a^\alpha=1$ follows. Then $x^2=x^{3\beta}$, which yields the contradiction $2=3\beta$. This proves that $x$ must be periodic, say with order $n$. Suppose $n$ is divisible by an odd prime $p$. Since $\langle a, x^p\rangle$ is not abelian, $x^2\in C_G(\langle a, x^p\rangle)< \langle a, x^p\rangle = \langle a\rangle\langle x^p\rangle$. As before, we can write $x^2=x^{p\beta}$ for some $\beta\in\Z$, that is, $p\beta=2+cn$ for some $c\in \Z$. But this is not possible since $p$ is odd and dividing $n$.

Let now $g\in C_G(\langle a,x\rangle)$. Since $\langle a,x\rangle$ is non-abelian,  $C_G(\langle a,x\rangle)<\langle a,x\rangle=\langle a\rangle\langle x\rangle$. Then we can write $g=a^\alpha x^\beta$ for suitable $\alpha,\beta\in \Z$. Thus $g=g^x=a^{-\alpha} x^\beta$, hence $a^{2\alpha}=1$ and so $a^\alpha=1$, that is, $g=x^\beta$. Since $a=a^g=a^{(x^\beta)}$, it follows  that $\beta$ is even, and therefore $g\in\langle x^2\rangle$. This completes the proof.
\end{proof}


The following results are easy, but useful, observations.

\begin{lemma}
\label{l:2-gen}
A group $G$ is in the class $\Y$ if and only if $C_G(\langle x,y\rangle)<\langle x,y\rangle$ for every pair of non-commuting elements $x,y\in G$.
\end{lemma}

\begin{corollary}
\label{c:minimal}
A finite group $G$ is in the class $\Y$ if and only if $C_G(K)<K$ for every minimal non-abelian subgroup $K$ of $G$.
\end{corollary}

\section{Infinite nilpotent $\Y$-groups}
\label{s:nilpotent}
Every abelian group lies in $\Y$, thus,  compared with $\X$-groups, the structure of $\Y$-groups which are soluble or nilpotent is less restricted. Our first result reduces the study of nilpotent $\Y$-groups to  finite nilpotent $\Y$-groups.

\begin{theorem} 
\label{t:nilpotent}
Every nilpotent $\Y$-group is abelian or finite.
\end{theorem}
\begin{proof}
First let $G$ be an $\Y$-group of class $c=2$; we show that $G$ is finite. Note that  $G$ is metahamiltonian, that is, all its non-abelian subgroups are normal:  if $H\leq G$ is non-abelian, then $H\unlhd G$ follows from $G'\leq Z(G)< H$. This implies that $G'$ is a finite $p$-group  
 (see, for example, \cite[Theorem 2.1]{Def13}).
Let $a,b\in G$ with $[a,b]\ne 1$ and define $H=\langle a,b\rangle$. Since $G\in\Y$, we have $C_G(H)=Z(H)$.  Note that if $x\in G$, then $|G:C_G(x)|=|x^G|=|\{g^{-1}xg\mid g\in G\}|=|\{[x,g]\mid g\in G\}|\leq |G'|$, hence \[|G:C_G(H)|=|G:(C_G(A)\cap C_G(b))|\leq |G:C_G(a)||G:C_G(b)|\leq |G'|^2,\]
so $|G:C_G(H)|$ is finite. Recall that $C_G(H)=Z(H)$ and that $H'\leq G'$ is  finite; thus, to prove that $G$ is finite, it suffices to prove that $|Z(H):H'|$ is finite. Recall that $G'$ is a finite abelian $p$-group, thus, for every positive integer $n$ coprime to $p$, the commutator formula in \cite[Corollary 1.1.7(ii)]{Lee02} yields  $[a^n,b^n]=[a,b]^n\ne 1$, in particular, $a^n,b^n\ne 1$. Since $G\in\Y$, this implies that  $Z(H)\leq \langle a^n,b^n\rangle$ for all such $n$. It follows that $Z(H)/H'$ lies in \[K=\bigcap\nolimits_{\gcd(n,p)=1} (H/H')^n=\bigcap\nolimits_{\gcd(n,p)=1} \langle a^n,b^n\rangle H'/H',\]where $(H/H')^n$ is the subgroup of $H/H'$ generated by all $n$-th powers. Note that $K$ is finite, as it is equal to the $p$-part of $H/H'$,
hence so is  $Z(H)/H'\leq K$. As discussed above, this implies that $G$ is finite. 
 
Now consider an $\Y$-group $G$ of nilpotency class  $c>2$ and use induction on $c$. Suppose, for a contradiction, that $G$ is infinite. For all  $a\in G\setminus G'$, the nilpotency class of $\langle a\rangle G'$ is less than $c$. If $\langle a\rangle G'$ is abelian for all $a\in G\setminus G'$, then $G'\leq Z(G)$ and $c\leq2$, which is a contradiction to our assumption that $c>2$. Thus, $H=\langle a\rangle G'$ is non-abelian for some  $a\in G\setminus G'$; by the induction hypothesis, $H$ is finite. This contradicts Lemma~\ref{l:normal}, thus $G$ is finite.
\end{proof} 

Let $G$ be a Chernikov $p$-group, that is,  $G$ is an extension of a direct product of a finite number $k$ of Pr\"ufer $p$-groups by a finite group; the number  $\delta(G)=k$ is an invariant of $G$.

\begin{proposition}
\label{p:locallynilpotent}
If $G$ is a non-abelian infinite locally nilpotent $\Y$-group, then $G$ is a Chernikov $p$-group with $\delta(G)\geq p-1$.
\end{proposition}
\begin{proof} Every finitely generated subgroup of $G$ is nilpotent, so by Theorem~\ref{t:nilpotent} it is either abelian or finite. This implies that all torsion-free elements of $G$ are central. It follows from \cite[Proposition~1]{Del07}) that $G$ is periodic, hence  $G$ is direct product of groups of prime-power order, see for instance \cite[Proposition~12.1.1]{Rob96}. In fact, only one prime can occur since $G$ is an $\Y$-group, that is, $G$ is a locally finite $p$-group for some prime $p$. Since $G$ is non-abelian, there exist non-commuting $a,b\in G$. The subgroup $H=\langle a,b\rangle$ is finite and non-abelian, thus $C_G(H)< H$ is finite and $G$ has finite Pr\"ufer rank by \cite[Theorem~5]{End10}.  By a result of Blackburn \cite{Bla62}, the group  $G$ is a Chernikov $p$-group. Finally, $\delta(G)\geq p-1$ by a result of Chernikov, see for instance \cite{Bla62}.
\end{proof}

As a consequence of Proposition~\ref{p:locallynilpotent} and \cite[12.2.5]{Rob96} we get the following corollary; recall that a group is hypercentral if  its upper central series terminates at the whole group.

\begin{corollary}
\label{c:hypercentral}
Every locally nilpotent $\Y$-group is hypercentral.
\end{corollary}

Theorem~\ref{t:nilpotent} and Proposition~\ref{p:locallynilpotent} reduce the study  of nilpotent $\Y$-groups to finite $p$-groups; we provide some results on these groups in Section~\ref{s:finite}.

\section{Infinite supersoluble $\Y$-groups}
\label{s:supersoluble}
A group is supersoluble if it admits a normal series with cyclic sections. In this section we describe the structure of infinite supersoluble $\Y$-groups completely. If such a group has no involutions, then it must be abelian.

\begin{theorem} 
\label{t:notranspositions}
An infinite supersoluble $\Y$-group with no involutions is abelian.
\end{theorem}
\begin{proof} Let $G$ be an infinite supersoluble $\Y$-group without elements of even order. Then the set  $T$ of all periodic elements in $G$ is a finite subgroup of $G$ by \cite[5.4.9]{Rob96}. As $T<G$, it easily follows that $G$ is generated by its aperiodic elements, hence $T\leq Z(G)$ by Corollary~\ref{c:aperiodic}. By a result of Zappa (see for instance \cite[5.4.8]{Rob96}), the quotient group $G/T$ has a finite series  
$$G_0/T=T/T\leq G_1/T\leq\dots\leq G_s/T\leq G/T,$$
where each $G_i$ is normal in $G$, each $G_{j+1}/G_j$ is infinite cyclic, and $G/G_s$ is a finite $2$-group. We claim that each $G_i\leq Z(G)$. Since $G_0=T$, this is true for $i=0$. By induction on $i$, let us assume that $G_i\leq Z(G)$ with $0\leq i\leq s-1$; then $G_{i+1}=\langle a \rangle G_i$ where $\langle a \rangle$ is infinite. Suppose, for a contradiction, that $G_{i+1}/G_i\not\leq Z(G/G_i)$. Then $(G/G_i)/C_{G/G_i}(G_{i+1}/G_i)$ has order~2, hence there exists $x\in G$ such that $a^xG_i=a^{-1}G_i$. It follows that $a^x=a^{-1}y$ for a suitable $y\in G_i$, hence $a^{(x^2)}=(a^{-1}y)^{-1}y=a$ and $a^{(x^2)}=a$. Now it follows from  $1=[a,x^2]=[a,x][a,x]^x$ that $[a,x]^x=[a,x]^{-1}$, and therefore $[a,x]=1$ by Lemma~\ref{l:inverting}, a contradiction. This shows that $G_{i+1}/G_i\leq Z(G/G_i)$, hence $G_{i+1}\leq Z_2(G)$. It follows readily that $\langle x\rangle G_{i+1}$ is nilpotent for all $x\in G$. Since $G_{i+1}$ is infinite, $\langle x\rangle G_{i+1}$ is abelian by Theorem~\ref{t:nilpotent}. This means that $G_{i+1}\leq C_G(x)$ for all $x\in G$, so $G_{i+1}\leq Z(G)$, and the claim is proved. In particular, $G_s\leq Z(G)$. Since $G/G_s$ is a finite 2-group (hence nilpotent), it follows that $G$ is nilpotent. Hence $G$ is abelian by Theorem~\ref{t:nilpotent}.
\end{proof}

Now we consider  infinite supersoluble $\Y$-groups with involutions.
\begin{theorem} 
\label{t:supersoluble}
Let $G$ be a non-abelian infinite supersoluble group. Then $G$ lies in $\Y$ if and only if the following holds:
$G=A\langle x\rangle$ with $x$ of order $2^n$ and  $A=\langle a_1\rangle\times\dots\times\langle a_t\rangle\times\langle d\rangle$ abelian with each $a_i$ of infinite or odd order, $d$ of order $2^h$ (with  $d^{(2^{h-1})}=x^{(2^{n-1})}$ if $h>0$), and  $a^x=a^{-1}$ for all $a\in A$.
\end{theorem}
\begin{proof}
Let $G$ be an infinite supersoluble $\Y$-group which is not abelian. By Zappa's theorem (see \cite[5.4.8]{Rob96}), there exists a non-trivial normal subgroup $C$ of $G$ such that $C$ has no involutions and $G/C$ is a finite 2-group. Thus $C$ is infinite, and therefore it is abelian by Theorem~\ref{t:notranspositions}. So $G/C$ is non-trivial. 

We claim that if $x\in G\setminus C_G(C)$ and $x^2\in C_G(C)$, then $a^x=a^{-1}$ for all $a\in C$. Indeed, from $x\notin C_G(C)$ it follows that there exists  $c\in C$ such that $c^x\ne c$. Since $x^2\in C_G(C)$ we get $1=[c,x^2]=[c,x]^x[c,x]$, hence $[c,x]^x=[c,x]^{-1}$. By Lemma~\ref{l:inverting},  the order of $x$ is a power of $2$ and $C_G(\langle[c,x],x\rangle)\leq\langle x^2\rangle$. Since $(cc^x)^x=cc^x$ it follows that $cc^x\in C_G(\langle[c,x],x\rangle)$. Thus $cc^x\in \langle x^2\rangle\cap C$, hence $cc^x=1$ and $c^x=c^{-1}$. Now let $a\in C$. Since $a^{(x^2)}=a$, we get $(aa^x)^x=aa^x$ and  $aa^x\in C_G(\langle c,x\rangle)\leq\langle x^2\rangle$. Therefore $aa^x\in C\cap \langle x^2\rangle=\{1\}$. Thus $a^x=a^{-1}$ and the claim is proved. Notice that this implies $Z(G)\leq \langle x^2\rangle$.

Since $G$ is supersoluble, there exists a $G$-invariant subgroup $S$ of $C$ which is infinite cyclic, say $S=\langle s\rangle$. Then $C_G(C)\le C_G(S)$. Suppose there exists $x\in C_G(S)\setminus C_G(C)$. Note that $G/C_G(C)$ is a finite $2$-group. Replacing $x$ by a suitable power, we may assume that $x^2\in C_G(C)$. But then the same argument as in the previous paragraph shows that $s^x=s^{-1}$, a contradiction. Thus $C_G(C)=C_G(S)$, and since $S$ is infinite cyclic, the group $G/C_G(C)$ is easily seen to be cyclic of order two.

Since $C_G(C)/C$ is a finite 2-group, hence nilpotent, it follows that $C_G(C)$ is nilpotent. Thus $C_G(C)$ is abelian by Theorem~\ref{t:nilpotent}. Write $B=C_G(C)$ and choose $x\in G$ so that $G=B\langle x\rangle$. As seen above, $x$ has order $2^n$ for some $n\geq 1$, and $a^x=a^{-1}$ for all $a\in C$. Since $G/B$ has order $2$ we get $x^2\in B$, and so $x^2\in Z(G)$. Therefore $Z(G)=\langle x^2\rangle$.

Our next aim is to prove that for all $a \in B$ there exists an integer $\alpha$ with $a^x=a^{-1}x^{4\alpha}$. Let $a\in B$ and $c\in C\setminus\{1\}$. From $a^{(x^2)}=a$ it follows that $(aa^x)^x=aa^x$, hence $aa^x\in C_G(\langle c,x\rangle)\leq\langle x^2\rangle$. Thus $aa^x=x^\beta$ for some even integer $\beta=2\gamma$. If $n=1$ then $x^\beta=1$, so $a^x=a^{-1}$ and the claim is proved. Thus it remains to consider the case $n>1$. If  $\gamma$ is odd, then  $(ax)^2=axax=x^{2(1+\gamma)}\in\langle x^4\rangle$. If $c\in C\setminus\{1\}$, then $c^{ax}=c^{-1}$, hence by Lemma~\ref{l:inverting} we get $C_G(\langle c,ax\rangle)\leq\langle (ax)^2\rangle\leq\langle x^4\rangle$. On the other hand $x^2\in C_G(\langle c,ax\rangle)$, so $x^2=1$, a contradiction since $n>1$. This proves that $\gamma$ is even, so  4 divides $\beta$, as required. Therefore  $\beta=4\alpha$ for some integer $\alpha$.

Now we consider  $A=\{c\in B\mid c^x=c^{-1}\}$, which is a non-trivial abelian normal subgroup of $G$. If $a\in B$, then $a^x=a^{-1}x^{4\alpha}$, and $(ax^{-2\alpha})^x=(ax^{-2\alpha})^{-1}$ proves $ax^{-2\alpha}\in A$, hence $a\in A\langle x\rangle$. This means that $G=B\langle x\rangle=A\langle x\rangle$. Write $A=\langle a_1\rangle\times\dots\times\langle a_t\rangle\times D$, where each $a_i$ has infinite or odd order and $D$ is a finite 2-group. 
If $D$ is non-trivial and $y\in D$ has order 2, then $y^x=y$, hence $y\in Z(G)$; this implies that $D$ has a unique element of order 2. Since $D$ is abelian, this means that $D$ is cyclic, say $D=\langle d\rangle$. Let $2^h$ be the order of $d$. If $h>0$, then $d^{(2^{h-1})}=x^{(2^{n-1})}$, hence the structure of $G$ is as required.

For the converse, let $G=A\langle x\rangle$ be as in the statement. Then $B=\langle x^2\rangle A$ is abelian and $|G:B|=2$. Thus, for every non-abelian $H\leq G$ there exists $y=x^ic\in H$ with $i$ odd and $c\in A$. Since $G=B\langle y\rangle$, we have $C_G(y)=C_B(y)\langle y\rangle=Z(G)\langle y\rangle=\langle x^2,y\rangle$. Note that $y^2=x^{2i}$, and so $\langle y^2\rangle=\langle x^2\rangle$.  Thus $C_G(H)\leq C_G(y)=\langle y\rangle\leq H$, hence $G\in\Y$, as desired.
\end{proof}

\section{On soluble $\Y$-groups}
\label{s:soluble}  
We now briefly discuss  soluble groups in $\Y$; we start with a result on the Fitting subgroup of an infinite soluble $\Y$-group.

\begin{theorem}
\label{t:fitting}
If $G\in\Y$ is infinite, then the Fitting subgroup of $G$ is abelian.
\end{theorem}
\begin{proof} Suppose the Fitting subgroup $F$ of $G$ is  non-abelian, and let $a,b\in F$ be non-commuting. Now $H=\langle a\rangle^G\langle b\rangle^G\unlhd G$ is normal, nilpotent, and non-abelian, so $H$ is finite by Theorem~\ref{t:nilpotent}; this contradicts  Lemma~\ref{l:normal}.
\end{proof}

In \cite{Del15} it has been proved that every infinite soluble $\X$-group is metabelian. However, it is not possible to bound the derived length of soluble $\Y$-groups, even in the torsion-free case. For example, the standard wreath product of $n$ copies of the infinite cyclic group is a finitely generated soluble commutative-transitive group, and therefore an $\Y$-group, with derived length~$n$, see \cite[Corollary~18]{Wu98}. Torsion-free polycyclic $\X$-groups are abelian, see \cite{Del15}; this result is no longer true for $\Y$-groups.

\begin{example}
\label{e:non-abelianpolycyclic}
Let $F=\langle a\rangle\times\langle b\rangle$ be free abelian of rank $2$ and $G=F\rtimes \langle c\rangle$, where $c$ is the automorphism of $F$ defined by $a^c=ab$ and $b^c=a$. Clearly,  $\langle c\rangle$ acts fixed-point-freely on $F$, hence $G$ is commutative-transitive, see\cite[Lemma~7]{Wu98}. Thus, $G$ is a torsion-free polycyclic $\Y$-group of derived length~$2$.
\end{example}
 
\begin{proposition}
\label{p:polycyclic}
Let $G\in \Y$ be soluble. If $G$ has a non-abelian subnormal subgroup which is polycyclic, then $G$ is polycyclic.
\end{proposition}
\begin{proof} 
If $H\leq G$ is non-abelian, polycyclic, and  subnormal in $G$, then $C_G(H)< H$ is polycyclic, and the result follows by \cite[Theorem~2.1]{End10}.
\end{proof}

Recall that an element $g$ of a group $G$ is a right Engel element if for each $x\in G$ there exists a positive integer $n$ such that the left-normed commutator $[g,{}_nx]=1$. Let $R(G)$ denote the set of all right Engel elements of $G$. It is shown in  \cite[12.3.2 (ii)]{Rob96} that  $R(G)$ contains the hypercenter of $G$; on the other hand, it is well known that $R(G)$ may be larger than the hypercenter of $G$ even when $G$ is soluble.

\begin{proposition}
\label{p:locallysoluble}
Let $G\in \Y$ be locally soluble. If $G$ has a non-abelian finite subgroup which is contained in $R(G)$, then $G$ is locally finite.
\end{proposition}
\begin{proof} 
Let $H\leq R(G)$ be non-abelian finite. 
Then $C_G(H)< H$ is finite as $G$ is an $\Y$-group, and hence $G$ is locally finite by \cite[Theorem~2.4]{End10}.
\end{proof}

\begin{corollary}
\label{c:locallyfinite}
Let $G\in \Y$ be locally soluble. If the hypercenter of $G$ contains a pair of non-commuting periodic elements, then $G$ is locally finite.
\end{corollary}
\begin{proof} Let $a$ and $b$ non-commuting periodic elements contained in the hypercenter of $G$. Then $H=\langle a,b\rangle$ is a finite non-abelian subgroup which is contained in $R(G)$, and the result follows from Proposition~\ref{p:locallysoluble}.
\end{proof}

\section{Finite $\Y$-groups}
\label{s:finite}
Here we study finite groups in $\Y$, in particular, $p$-groups. The proof of the following preliminary lemma is similar to the proof of  \cite[Lemma 2.1]{Del13}.

\begin{lemma}\label{l:finbasic1}
Let $\mathcal{F}$ be a subgroup closed class of finite groups. Then $\mathcal{F}\subseteq\Y$ if and only if for every $G\in \mathcal{F}$ the following holds: if $K\leq G$ is non-abelian, then $Z(G)\leq K$.
\end{lemma}
\begin{proof}
If $\mathcal{F}\subseteq \Y$, then every non-abelian $K\leq G$ of $G\in\mathcal{F}$ satisfies $Z(G)\leq C_G(K)\leq K$. For the converse, denote by $\mathcal{F}_n$ the subset of $\mathcal{F}$ of groups  of order $n$, and proceed by induction on $n$. Clearly, $\mathcal{F}_1\subseteq \Y$, so we may assume that $n>1$. Let $G\in\mathcal{F}_n$ and let $K\leq G$ be non-abelian. Let $z\in C_G(K)$ and define $H=\langle K,z\rangle$; note that $z\in C_H(K)$. If $H = G$, then $z\in Z(G)$, hence $z\in K$ by assumption. If $H<G$, then $H\in \mathcal{F}_m$ for some $m<n$, hence $H\in\Y$ by the induction hypothesis, and therefore $z\in C_H(K)\leq K$. In conclusion, $C_G(K)\leq K$ for every non-abelian $K\leq G$, hence $G\in\Y$.
\end{proof}

\begin{lemma}
\label{l:finbasic2}
Let $\mathfrak{X}$ be a subgroup closed class of finite groups. Suppose that $Z(G)\le \Phi (G)$ for every non-abelian $G\in\mathfrak{X}$. Then $\mathfrak{X}\subseteq \Y$.
\end{lemma}
\begin{proof}
Let $G\in\mathfrak{X}\setminus\Y$ be of smallest possible order. Denote the family of all subgroups of $G$ by $\mathcal{F}$. Note that every proper subgroup of $G$ is in $\Y$. Let  $K\in\mathcal{F}$, and let $H$ be a non-abelian subgroup of $K$. We wish to prove that $Z(K)\le H$. If $K\neq G$, then $K\in\Y$, therefore
$Z(K)\le C_K(H)\le H$. If $K=G$, we can assume without loss of generality that $H\neq G$. Let $M$ be a maximal subgroup of $G$ containing $H$. As $M\in\Y$, we conclude that $C_M(H)\le H$. Since $Z(G)\le\Phi(G)$, it follows that $Z(G)\le C_M(H)\le H$, as required. Now  Lemma \ref{l:finbasic1} yields $\mathcal{F}\subseteq \Y$, which contradicts our assumption. 
\end{proof}

\begin{corollary}\label{c:finbasic3}
A non-abelian finite group $L$ lies in $\Y$ if and only if all its maximal subgroups lie in $\Y$ and $Z(L)\leq \Phi(L)$.
\end{corollary}
\begin{proof}
Clearly, if $L\in\Y$ and $M<L$ is maximal and non-abelian, then $M\in Y$, hence $Z(L)\leq C_L(M)\leq M$. If $M$ is abelian, then $Z(L)\leq M$, since otherwise $L=MZ(L)$ is abelian, a contradiction. 

Now consider the converse. Suppose the claim is not true and choose a minimal counterexample, that is, a group $L\notin \Y$ of smallest possible order such that $Z(L)\leq \Phi(L)$ and $M\in\Y$ for all maximal subgroups $M<L$. Let $\mathcal{F}$ be the set of all subgroups of $L$. Let $G\in \mathcal{F}$ and $K\leq G$ be non-abelian. We aim to show that $Z(G)\leq K$. If this holds for all such $G$ and $K$, then Lemma \ref{l:finbasic1} proves that  $L\in\mathcal{F}\subseteq \Y$, a contradiction. If $G<L$, then $G\leq M$ for some non-abelian maximal subgroup $M<L$. By assumption, $M\in\Y$, hence $G\in\Y$ and $Z(G)\leq C_G(K)\leq K$. Now consider $G=L$. If $K=G$, then $Z(G)\leq K$, so let $K\leq M<G$ for some non-abelian maximal subgroup $M<G$. By assumption, $Z(G)\leq M\in\Y$, hence $Z(G)\leq C_M(K)\leq K$.
\end{proof}


\subsection{Finite $p$-groups in $\Y$}
Motivated by Section \ref{s:nilpotent}, we now  concentrate on finite $p$-groups. The next two lemmas give the well-known characterizations of minimal non-abelian finite $p$-groups, see \cite[Lemma 2.2]{Xu08} and  \cite[Exercise 8a, p.\ 29]{Ber09}.

\begin{lemma}
\label{l:minnonab}
Let $G$ be a finite $p$-group. The following are equivalent:
\begin{ithm}
\item $G$ is minimal non-abelian.
\item $\d(G)=2$ and $|G'|=p$.
\item $\d(G)=2$ and $Z(G)=\Phi(G)$.
\end{ithm}
\end{lemma}

\begin{lemma}\label{lem:minnonab}
Every minimal non-abelian $p$-group is isomorphic to one of:
\begin{ithm}
\item $K_1=Q_8$,
\item $K_2=\langle a,b\mid a^{p^m}=b^{p^n}=1,a^b=a^{1+p^{m-1}}\rangle$ with $m\geq 2$ and $n\geq 1$; this group is metacyclic of order $p^{n+m}$,
\item $K_3=\langle a,b,c\mid a^{p^m}=b^{p^n}=c^p=1, [a,b]=c, [c,a]=[c,b]=1\rangle$ with $m+n>2$ if $p=2$; this group is not metacyclic of order $p^{m+n+1}$, and the derived subgroup $K_3'$ is a maximal cyclic subgroup.
\end{ithm}
\end{lemma}

Recall that a finite $p$-group $G$ lies in $\Y$ if and only if every minimal non-abelian $K\leq G$ satisfies $C_G(K)\leq K$; in particular, $Z(G)\leq K$. Since minimal non-abelian $p$-groups are classified, this poses strict conditions on $Z(G)$. We denote by $\Omega_1(G)$ the subgroup of $G$ generated by all elements of order $p$, and $G^p$ is the subgroup of $G$ generated by all $p$-th powers of elements of $G$. As usual, $C_n$ denotes a cyclic group of order $n$.

\begin{lemma}\label{l:centminnonab}
If $K$ is a minimal non-abelian $p$-group, then $Z(K)\cong C_{p^i}\times C_{p^j}$ or $Z(K)\cong C_p\times C_{p^i}\times C_{p^j}$
for some $i,j\geq 0$, and $\Omega_1(Z(K))$ is elementary abelian of rank 1, 2, or 3.
\end{lemma}
\begin{proof}
Let $K_1,K_2,K_3$ be as in Lemma \ref{lem:minnonab}. If $K\cong K_1$, then $Z(G)\cong C_2$. If $K\cong K_2$, then $Z(K)\cong\langle a^p,b^p\rangle$; if $K\cong K_3$, then $Z(K)\cong \langle a^p, b^p,c\rangle$
\end{proof}
The next result deals with metacyclic $p$-groups. We note that not all metacyclic groups are in $\Y$, since one can easily show, for example, that $D_{12}\notin\Y$. The situation is quite different if we restrict to $p$-groups.

\begin{proposition}
\label{p:metacyclic}
Every metacyclic $p$-group is in $\Y$.
\end{proposition}
\begin{proof}
Let $\mathfrak{X}$ be the class of  metacyclic $p$-groups, and let $G\in\mathfrak{X}$ be non-abelian. Note that all subgroups of $G$ are also metacyclic. By Lemma \ref{l:finbasic2} it suffices to prove that $Z(G)$ is contained in $\Phi(G)$. By \cite[Theorem 3.1]{Kin73}, the group $G$ can be given by a reduced presentation
$$G=\langle a,b\mid a^{p^m}=1,\; b^{p^n}=a^{p^{m-s}},\; a^b=a^{\varepsilon+p^{m-c}}\rangle,$$
where the precise numerical restrictions on $m$, $n$, $s$ and $c$ are given in \cite[Theorem 3.1]{Kin73}, and $\varepsilon =1$ when $p>2$, and $\varepsilon=\pm 1$ when $p=2$. Furthermore, it is clear that $\Phi (G)=\langle a^p,b^p\rangle$, and \cite[Proposition 4.10]{Kin73} gives $Z(G)=\langle a^{p^u},b^{p^v}\rangle$, where $u=v=c$ when $\varepsilon=1$, and $u=m-1$, $v=\max\{ 1,c\}$ when $\varepsilon=-1$. From here the assertion easily follows. 
\end{proof}

\begin{lemma}
Let $G\in \Y$ be a $p$-group. If $g\in G \setminus \Phi(G)$, then $C_G(g)$ is abelian.
\end{lemma}
\begin{proof}
Let $M$ be a maximal subgroup of $G$ such that $G =\langle g,M\rangle$. First, we show that $K = C_M(g)$ is abelian. If this is not the case, then, since $G, M\in \Y$, we have
$C_M(K) = C_G(K) = Z(K)$. By definition, $g$ commutes with $K$, so $g\in C_G(K)\leq Z(K)$, which is not possible since $K\leq M$ and $g\notin M$. Thus, $K=C_M(g)$ must be abelian. Note that every $h\in G$ can be written as $h=g^j m$ for some $m\in M$ and
$j=0,..,p-1$. Now $h\in C_G(g)$ if and only if   $g^{j+1} m = g g^jm = g^jmg = g^{j+1} m^g$, if and only if  $m^g = m$, if and only if $m\in C_M(g) = K$. This proves that $C_G(g) = \langle g,K\rangle$. As shown above, $K$ is abelian, hence also  $C_G(g)$ is abelian.
\end{proof}

\begin{lemma}\label{l:smallppower}
All $p$-groups of order $p$, $p^2$, and $p^3$ lie in $\Y$. A $p$-group of order $p^4$ lies in $\Y$ if and only if it is abelian, has maximal class, or $\Phi(G)=Z(G)$.
\end{lemma}
\begin{proof}
Groups of order $p^n$ with $n\leq 3$ are abelian or minimal non-abelian, hence lie in $\Y$. Let $G$ be a group of order $p^4$. If $G$ is abelian, then $G\in \Y$. If $G$ has maximal class, then every maximal subgroup $M<G$ satisfies $Z(G)<G'<M$; clearly, $M\in\Y$ since $|M|=p^3$, hence $G\in\Y$ by Corollary \ref{c:finbasic3}. Now suppose $G$ has nilpotency class 2, hence $G>Z(G)\geq G'\geq 1$. Since every maximal subgroup $M<G$ lies in $\Y$, we have $G\in\Y$ if and only if $Z(G)\leq \Phi(G)$. Since $G$ is not abelian, we have $|G:\Phi(G)|\geq p^2$. If $Z(G)<\Phi(G)$, then we must have $G>\Phi(G)>Z(G)=G'>1$; but $Z(G)=G'$ for a non-abelian group $G$ of order $p^4$ implies that $G/Z(G)$ has order $p^2$, see \cite[p.\ 11, Exercise 40]{Ber09}, a contradiction. Thus, $\Phi(G)=Z(G)$.
\end{proof}


We now consider $p$-groups of maximal class in more detail; the following remark recalls some important properties.

\begin{remark}\label{r:maxclass} Let $G$ be  a group of order $p^n$ of maximal class, and suppose $n\geq 4$. By  \cite[p. 56]{Lee02}, there exists a chief series $G>P_1>\ldots>P_n=1$ with $P_i=P_i(G)=\gamma_i(G)$ for $i\geq 2$, and $P_1=P_1(G)=C_G(P_2/P_4)$; we define $P_m=P_m(G)=1$ if $m>n$. Note that $[P_i,P_j]\leq P_{i+j}$ for all $i,j\geq 1$. We note that $P_2,\ldots,P_n$ are the unique normal subgroups of $G$ of index greater than $p$, see \cite[Proposition 3.1.2]{Lee02}.
Following \cite[Definition 3.2.1]{Lee02}, we say the {\it degree of commutativity} of a $p$-group of maximal class is the largest integer $\ell$ with the property that  $[P_i,P_j]\le P_{i+j+\ell}$ for all $i,j\ge 1$ if $P_1$ is not abelian, and $\ell =n-3$ if $P_1$ is abelian. If $n>p+1$, then  $G$ has positive degree of commutativity, see \cite[Theorem 3.3.5]{Lee02}.  If $s\in G\setminus P_1$ and $s_1\in P_1\setminus P_2$, then $s$ and $s_1$ generate $G$; for $i\leq n$ define $s_i=[s_{i-1},s]$. If $G$ has positive degree of commutativity, then $s_i\in P_i\setminus P_{i+1}$ for all $i<n$, see \cite[Lemma 3.2.4]{Lee02}.
\end{remark}

\begin{lemma}
\label{l:pmaxclass}
If a $p$-group $G$ of maximal class has an abelian maximal subgroup, then $G\in\Y$.
\end{lemma}
\begin{proof}
Let $G\notin\Y$ be a $p$-group of maximal class that has an abelian maximal subgroup $A$, and suppose that it is of smallest possible order. Let $H$ be a proper non-abelian subgroup of $G$. It follows from \cite[p.\ 27, Exercise 4]{Ber09} that $H$ is also of maximal class; now note that $H\cap A$ is an abelian maximal subgroup of $H$, since $p=|G:A|=|HA:A|=|H:H\cap A|$. Thus, all proper subgroups of $G$ belong to $\Y$. Taking $\mathcal{F}$ to be the family of all subgroups of $G$, we get a contradiction similarly as in the proof of Lemma \ref{l:finbasic2}.
\end{proof}

There exist $3$-groups of maximal class all of whose  maximal subgroups are non-abelian, see \cite{Bla58}. Nevertheless, we can prove the following.

\begin{lemma}
\label{l:3maxclass}
The $2$- and $3$-groups of maximal class lie in $\Y$.
\end{lemma}
\begin{proof}
  The $2$-groups of maximal class are classified and all metacyclic, see \cite{Bla58}, hence they lie in $\Y$ by Proposition \ref{p:metacyclic}. Now let $G$ be a 3-group of maximal class of order $3^n$ with chief series $G>P_1>\ldots>P_n=1$ as defined above. We proceed by induction on $n$. It follows from Lemma \ref{l:smallppower} that $G\in\Y$ if $n\leq 4$, so let $n\geq 5$ in the following.  By Lemma \ref{l:pmaxclass}, we can assume that all maximal subgroups are non-abelian. It follows from \cite[Corollary 3.3.6]{Lee02} that every maximal subgroup $M<G$ either has maximal class, or  $M=P_1$. In the first case, $M\in\Y$ by the induction hypothesis; also, $Z(G)\leq M$ since otherwise $G=MZ(G)$, and $Z(G)Z(M)\leq Z(G)$ yields a contradiction to $|Z(G)|=3$. If $M=P_1$, then $|M'|=3$ by  \cite[Theorem 3.4.3]{Lee02}. It follows from \cite[Corollary 3.3.6]{Lee02} that $M^3=P_3$. Since $M'\leq P_3$, this yields $\Phi(M)=M'M^{3}=P_3$, so $M/\Phi(M)=P_1/P_3$. This proves that $M$ is a 2-generator group, hence minimal non-abelian by Lemma \ref{l:minnonab}, and $M\in \Y$. Now Corollary \ref{c:finbasic3} proves that $G\in\Y$.
\end{proof}


\begin{theorem}\label{t:maxclass}
Let $G$ be a $p$-group of maximal class of order $p^n$. If $p\in\{2,3\}$ or $n\leq 3$, then $G\in\Y$. If $p\geq 5$ and $n\geq 4$, then $G\in\Y$ if and only if its 2-step centralizer $P_1(G)$ is abelian.
\end{theorem}

\begin{proof}
By Lemmas \ref{l:3maxclass} and \ref{l:smallppower}, it suffices to consider $p\geq 5$ and $n\geq 4$. For $n=4$ the claim follows from the classification in \cite[Satz 12.6]{Hup67}, so let $n\geq 5$. Let $G$ be of maximal class and define $G>P_1>\ldots>P_n=1$  and $s,s_1,\ldots,s_{n-1}$ as in Remark \ref{r:maxclass}. Clearly, if $P_1$ is abelian, then $G$ has an abelian maximal subgroup, hence $G\in\Y$ by Lemma \ref{l:pmaxclass}. 

For the converse, suppose that $G$ is a counterexample of smallest order. Choose $s\in G\setminus P_1(G)$ with $|C_G(s)|=p^2$ and $s^p\in Z(G)$, which is possible by \cite[Hilfssatz III.14.13]{Hup67}, and consider $M=\langle s,P_2\rangle$. It is easy to see that $M\in\Y$ has maximal class. Since $[P_2,P_3]\leq P_5$, it follows that $P_1(M)=P_2$. If $P_2$ is nonabelian, then this yields a contradiction to our choice of $G$; thus $P_2=G'$ is abelian and $G$ is metabelian. It follows from \cite[Theorem 2.10]{Bla58} that $G$ has positive degree of commutativity $\ell>0$, cf.\ \cite[p.\ 74]{Bla58}; in particular, if $n>p+1$, then $\ell\geq n-p-1$, see \cite[Theorem 3.10]{Bla58}. Recall that we assume  $P_1$ is nonabelian, thus $Z(P_1)=P_m$ for some $m\in\{2,\ldots,n-2\}$. Since $[P_1,G,P_{m-1}]=1=[G,P_{m-1},P_1]$, the three-subgroup lemma \cite[Proposition 1.1.8]{Lee02} shows that $[P_{m-1},P_1,G]=1$, hence $[P_1,P_{m-1}]=P_{n-1}=Z(G)$. The same argument and an induction can be used to show that $[P_1,P_{m-i}]\leq P_{n-i}$ for all $i=1,\ldots,m-1$, which implies that $\ell=n-m-1$. Define $H=\langle x,y\rangle$ for some $x\in P_1\setminus P_2$ and $y\in P_{m-1}\setminus P_m$. Note that $[x,y]\in P_{1+m-1+\ell}=P_{n-1}$, and $[x,y]\ne 1$ since $P_2$ is abelian and $P_m=Z(P_1)$; thus, $H'=Z(G)=P_{n-1}$. Since $G\in\Y$, it follows that $Z(P_1)\leq H$, thus $P_{m-1}=\langle y,P_m\rangle \leq H$. Since $\ell=n-m-1>0$, and so $n\geq m+2$, this implies that $P_{n-2}\leq P_{m-1}\leq H$.

First, let $n\leq p+1$, so that  $G/Z(G)$ has exponent $p$ by \cite[Proposition 3.3.2]{Lee02}. Since $H$ is a 2-generator group and $Z(G)=H'$, it follows that $Z(G)=\Phi(H)$, hence  $|H/Z(G)|\leq p^2$. Now $Z(G)<P_{n-2}\leq Z(H)\leq H$ implies that  $|H:Z(H)|\leq p$, so $H$ is abelian. This is a contradiction. 

Second, consider $n>p+1$ and $\ell\geq 2$. Now  \cite[Corollary 3.3.6]{Lee02} shows that $P_{m-1}/P_{m+p-2}$ is a subgroup of $\Omega_1(H/P_{m+p-2})$ of  order $\min\{p^{p-1},p^{n-m+1}\}=\min\{p^{p-1},p^{\ell+2}\}\geq p^4$. But $H/P_{m+p-2}$ is a 2-generator $p$-group whose derived subgroup has order at most $p$, and so $|\Omega_1(H/P_{m+p-2})|\leq p^3$, a contradiction.

Lastly, consider $n>p+1$ and $\ell=1$. As mentioned above, $\ell\geq n-p-1$, which implies that $n=p+2$ and $m=p$.  Thus, $Z(P_1)=P_p$, and   \cite[Corollary 3.3.6]{Lee02} shows that  $x^p\in P_p\setminus P_{p+1}$ and $W=P_{p-2}$ has exponent $p$; recall $p\geq 5$. Since $P_2=G'$ is abelian and $x\in P_1\setminus P_2$, it follows that $C_W(x)=Z(P_1)\cap W =P_p$, hence  $|\{[x,w]\mid w\in W\}|=|\{x^w\mid w\in W\}|=|W:C_W(x)|=p^2$. Together with $\{[x,w]\mid w\in W\}\subseteq P_p$, this implies that $P_p=\{[x,w]\mid w\in W\}$. In particular, there is $w\in W$ with $[x,w]=x^p$, which implies that $J=\langle x,w\rangle$ is non-abelian with order $p^3$ and exponent $p^2$, and so $|\Omega_1(J)|=p^2$. It follows from  $G\in \Y$ that $P_p\leq J$, thus $\langle w,P_p\rangle \leq \Omega_1(J)$ and $|\Omega_1(J)|\geq p^3$. This final contradiction completes the proof.
\end{proof}


We end this section with a classification of the $p$-groups in $\Y$ of exponent~$p$.

\begin{theorem}
Let $G\in\Y$ be a finite $p$-group of exponent $p$. If $|G|>p^p$, then $G$ is elementary abelian. Otherwise, either $G$ is elementary abelian, or  $G$ has maximal class and an elementary abelian subgroup of index $p$.
\end{theorem}
\begin{proof}
Clearly, if $G$ is abelian, then $G$ is elementary abelian. Thus, in the following, suppose that $G$ is non-abelian. By Lemma \ref{lem:minnonab}, if $p>2$, then every minimal non-abelian $K\leq G$ must be extra-special of order $p^3$ and $Z(G)=Z(K)\cong C_p$; if $p=2$, then there is no minimal non-abelian subgroup of exponent $2$, hence $G$ is elementary abelian. Thus, in the following let $p>2$ and $|Z(G)|=p$; we prove the assertion by induction on the order of $G$. 

By Lemma \ref{l:smallppower}, our claim is true if $|G|$ divides $p^3$; if $|G|=p^4$, then the claim follows from the known classification of groups of order $p^4$, see \cite[Satz 12.6]{Hup67}. So in the following we discuss the case $n\geq 5$. By the induction hypothesis, each maximal subgroup $M<G$ is either elementary abelian, or has maximal class and $M\cong \langle h\rangle \ltimes C_p^{n-2}$. Note that the latter can only happen if  $n\leq  p+1$, since otherwise $|M|=p^{n-1}>p^p$ and then the induction hypothesis forces  $M$ to be  elementary abelian.

Suppose $G$ has an abelian maximal subgroup $M$. Since $|Z(G)|=p$, it follows from \cite[Exercise 4, p.\ 27]{Ber09} that $G$ has maximal class. Define $G>P_1>\ldots>P_{n}=1$ and $s,s_1,\ldots,s_{n-1}$ as in Remark \ref{r:maxclass}.  Note that $P_1=M$ since $M$ is abelian and $P_1=C_G(P_2/P_4)$, hence $G$ has positive degree of commutativity. If $n>p$, then $s_p\in P_p\setminus P_{p+1}$ by \cite[Lemma 3.2.4]{Lee02}. Now  \cite[Corollary 1.1.7(i)]{Lee02} yields $(ss_1)^p=s_p\ne 1$, which is a contradiction to $(ss_1)^p=1$. This proves that if $|G|>p^p$ and $G$ has a maximal subgroup which is elementary abelian, then $G$ is elementary abelian.

Now suppose  $G$ has no elementary abelian maximal subgroup, and so $n\leq p+1$ as shown above. If $M<G$ is maximal, then $M$ has maximal class and an elementary abelian subgroup $N<M$ of order $p^{n-2}$. Observe that $N=P_1(M)$ is characteristic in $M$, hence $N\unlhd G$. Let $M^\ast\ne M$ be maximal subgroup of $G$, and define $N^\ast=P_1(M^\ast)$. Note that $N^\ast\unlhd G$ is abelian, and
\[(\ast)\quad |NN^\ast:N\cap N^\ast|=|NN^\ast:N||N:N\cap N^\ast|=|NN^\ast:N||NN^\ast:N^\ast|.\]
Suppose $N\ne N^\ast$, so $NN^\ast$ has index $1$ or $p$ in $G$. If $G=NN^\ast$, then $G'=[N,N^\ast]\leq N\cap N^\ast$ since $N$ and $N^\ast$ are abelian and normal in $G$;  now $|G:G'|\geq p^4$ by ($\ast$), contradicting $|G:M'|=p^3$. Thus, if $N\ne N^\ast$, then $L=NN^\ast< G$ is maximal. Note that $N\cap N^\ast\leq Z(L)$, and $|L:Z(L)|\leq |L:N\cap N^\ast|= p^2$ by ($\ast$). Since $L$ has maximal class, $|L|\leq p^3$, which contradicts $|G|\geq p^5$. In conclusion, we have proved $N=N^\ast$; in particular, $N$ is contained in every maximal subgroup of $G$, and so $N=P_1(M)=\Phi(G)=\gamma_2(G)$. An induction on $i$ now proves that $\gamma_i(G)=\gamma_{i-1}(M)$ for all $i\geq 3$, hence $G$ has maximal class. By Theorem \ref{t:maxclass}, the group $G$ has an abelian maximal subgroup -- a contradiction to our assumption. In conclusion, $G$ has an elementary abelian maximal subgroup, and the claim follows.
\end{proof}
   
\section*{Acknowledgement}
\noindent The authors wish to express their gratitude to the anonymous referee(s) for their very thorough reading, and for suggesting shorter proofs. In particular, the proofs of  Theorem \ref{t:nilpotent} and Theorem \ref{t:maxclass} have been shortened significantly due to alternative arguments given by the referee(s).

\section*{References}


\begin{thebibliography}{99}  
\bibitem{Ber09} Y.\ Berkovich, 
Groups of prime power order, Vol.\ 1, 
Walter de Gruyter GmbH \& Co. KG, Berlin, 2008.

\bibitem{Bla58} N.\ Blackburn,
{\em On a special class of $p$-groups}, 
Acta Math.\ {\bf 100} (1958), 49--92.

\bibitem{Bla62} N.\ Blackburn,
{\em Some remarks on \v Cernikov $p$-groups}, 
Illinois J.\ Math.\ {\bf6} (1962), 421--433. 
 
\bibitem{Del07} C.\ Delizia, 
{\em Some remarks on aperiodic elements in locally nilpotent groups}, 
Int.\ J.\ Algebra {\bf 1} (2007), no. 7, 311--315.

\bibitem{Del13} C.\ Delizia, U.\ Jezernik, P.\ Moravec and C.\ Nicotera, 
{\em Groups in which every non-cyclic subgroup contains its centralizer},
J.\ Algebra Appl. {\bf13} (2014), no.\ 5, 1350154 (11 pages).

\bibitem{Del15} C.\ Delizia, U.\ Jezernik, P.\ Moravec, C.\ Nicotera, C.\ Parker, 
{\em Locally finite groups in which every non-cyclic subgroup is self-centralizing}, 
submitted.

\bibitem{End10} G.\ Endimioni and C.\ Sica, 
{\em Centralizer of Engel elements in a group},
Algebra Colloq.\ {\bf17} (2010), no. 3, 487--494.

\bibitem{Def13}
 M. De Falco, F. De Giovanni, and C. Musella, {\it Metahamiltonian groups and related topics},
  Int. J. Group Theory {\bf 2} (2013), no. 1, 117--129.

\bibitem{Hup67} B.\ Huppert, Endliche Gruppen I, Springer Verlag, 1967.

\bibitem{Kin73} B.\ W.\ King,
{\em Presentations of metacyclic groups},
Bull.\ Austral.\ Math.\ Soc.\ {\bf 8} (1973), 103--131.

\bibitem{Lee02} C.\ R.\ Leedham-Green and S.\ McKay,
The structure of groups of prime power order,
Oxford University Press, 2002.

\bibitem{Rob96} D.\ J.\ S.\ Robinson,  A course in the theory of groups, 2nd Edition, Springer-Verlag, 1996.
 



\bibitem{Wu98} Y.\ F.\ Wu,
{\it Groups in which commutativity is a transitive relation},
J.\ Algebra {\bf 207} (1998), 165--181.

\bibitem{Xu08} X.\ Xu, L.\ An and Q.\ Zhang,
{\em Finite p-groups all of whose non-abelian proper subgroups are generated by two elements},
J.\ Algebra {\bf 319} (2008), 3603--3620.   
  

\end{thebibliography}
\end{document}